\documentclass[10pt]{amsart}

\usepackage{geometry,graphicx,amssymb,amsmath,amsbsy,eucal,amsfonts,mathrsfs,amscd,bm,tcolorbox}

\usepackage[all]{xy}
\usepackage{tikz}
\usetikzlibrary{arrows,shapes,calc,snakes}
\usetikzlibrary{shapes,snakes}
\usepackage{tikz-3dplot}
\usetikzlibrary{intersections}
\usepackage{tkz-euclide}
\usetkzobj{all}
\usetikzlibrary{angles}
\usepackage{pgfplots}

\numberwithin{equation}{section}

\allowdisplaybreaks[3]

\newtheorem{theorem}{Theorem}[section]
\newtheorem{lemma}[theorem]{Lemma}

\newtheorem{corollary}[theorem]{Corollary}
\newtheorem{proposition}[theorem]{Proposition}
\theoremstyle{definition}

\theoremstyle{remark}
\newtheorem{remark}[theorem]{Remark}

\newcommand{\R}{\mathbb{R}}

\newcommand{\eps}{\varepsilon}

\newcommand{\mct}{\mathcal{T}_h}

\newcommand{\jump}[1]{\left[\hspace{-0.025in}\left[#1\right]\hspace{-0.025in}\right]}

\newcommand{\curl}{{\ensuremath\mathop{\mathrm{curl}\,}}}
\newcommand{\dive}{{\ensuremath\mathop{\mathrm{div}\,}}}

\newcommand{\pol}{\EuScript{P}}
\newcommand{\bpol}{\boldsymbol{\pol}}
\newcommand{\bld}[1]{\boldsymbol{#1}}

\newcommand{\ff}{{\bld{f}}}

\newcommand{\bv}{\bld{v}}
\newcommand{\bw}{\bld{w}}

\newcommand{\bn}{\bld{n}}
\newcommand{\bu}{\bld{u}}
\newcommand{\bff}{\bld{f}}

\newcommand{\bV}{\bld{V}}

\newcommand{\bPi}{\bld{\Pi}}

\newcommand{\bH}{\bld{H}}

\newcommand{\bx}{\bld{x}}

\newcommand{\bL}{\bld{L}}

\newcommand{\Om}{\Omega}

\newcommand{\bbeta}{{\bm \beta}}

\newcommand{\be}{\bm e}

\title[Numerical analysis of a linearised model of inviscid
incompressible flow]{Well-posedness and H(div)-conforming finite element approximation of a linearised model for inviscid
incompressible flow}

\author[Gabriel Barrenechea]{Gabriel Barrenechea}
\address{Department of Mathematics and
    Statistics, University of Strathclyde, 26 Richmond Street,
    Glasgow, G1 1XH United Kingdom}
\email{gabriel.barrenechea@strath.ac.uk}

\author[Erik Burman]{Erik Burman}
\address{Department of Mathematics, University College London, London, UK-WC1E  6BT, United Kingdom}
\email{e.burman@ucl.ac.uk}

\author[Johnny Guzman]{Johnny Guzman}
\address{Division of Applied Mathematics
Brown University
Box F
182 George Street
Providence, RI 02912}
\email{johnny\_guzman@brown.edu}

\begin{document}

\begin{abstract}
We consider a linearised model of incompressible inviscid flow. Using
a regularisation based on the Hodge Laplacian we prove existence and uniqueness of weak
solutions for smooth domains. The model problem is then discretised
using H(div)-conforming finite element methods, for which
we prove error estimates for the velocity approximation in the $L^2$-norm of order  $O(h^{k+\frac12})$. We also prove error estimates for
the pressure error in the $L^2$-norm.
\end{abstract}

\maketitle

\section{Introduction}
The use of H(div)-conforming finite element methods for the approximation of incompressible
flow at high Reynolds number has been
receiving increasing attention from the research community
recently \cite{GSS17,NC18,SL18}. By construction such methods can satisfy the divergence-free
condition exactly. The lack of $H^1$-conformity is handled using
techniques drawing on ideas from discontinuous Galerkin methods \cite{GRW05}, resulting in
several possible different choices for the discretisation of the
transport term and the viscous term. For the former one may either design an energy conserving method
using central fluxes, or one may opt for a dissipative alternative in
the form of upwind fluxes. The latter were shown in \cite{GSS17} to be
more robust than the former, as is the case for  discontinuous Galerkin (DG)
methods. For DG-methods applied to scalar problems it is well known that thanks to the
dissipative properties of the upwind flux one may prove an error
estimate in the $L^2$-norm, of the form (see, e.g., \cite{JP86})
\begin{equation}\label{eq:L2error}
\|u - u_h\|_{L^2(\Omega)} \leq C h^{k+\frac12} |u|_{H^{k+1}(\Omega)},
\end{equation}
where $u$ is the exact solution, $u_h$ its DG-approximation, $\Omega
\subset \mathbb{R}^d$, $d=2,3$.
is the computational domain, $h$ the mesh parameter, and finally $k$
the polynomial degree of the approximation space.
On special meshes one can in fact prove optimal estimates with rate $h^{k+1}$ for upwind DG methods applied to scalar problems 
\cite{cockburn2008optimal, richter1988optimal}. However, as it is
shown in \cite{peterson1991note}, the result \eqref{eq:L2error} is
sharp on general meshes.

Estimates of the type  \eqref{eq:L2error} are also the best that are known for either
stabilised conforming finite element approximations, or fully DG methods, of laminar
solutions of the Navier-Stokes' equations in the high Reynolds number regime
\cite{BF07,HS90}, or the incompressible Euler equations \cite{JS86,Bu15}. The robustness of the H(div)-conforming
elements in the case of vanishing viscosity was shown in \cite{KS11}
for the case of the Brinkman problem, i.e. without the convection
terms. 
Despite all the work
quoted above, there seems
to be no proof of an error estimate of the type \eqref{eq:L2error} for
finite element methods using H(div)-conforming elements applied to
incompressible flow problems (see the discussion in \cite{SL18,NC18}).

The purpose of this work is to fill the gap mentioned in the last paragraph. That is, proving an estimate of the type
\eqref{eq:L2error} for finite element methods approximating a stationary linearised model of inviscid
flow and using H(div)-conforming approximation spaces for the velocity
approximation. Both the spaces designed by Raviart and Thomas \cite{RT}
and by Brezzi, Douglas and Marini \cite{BDM} enter the framework. As stabilising fluxes, these need to be either upwind, or,
in case of central fluxes, an additional penalty term on the jump of the tangential component
of the velocity needs to be added.
In the particular case in which the velocity is approximated using the Raviart-Thomas space
we also prove a convergence result for the 
pressure error, showing that the approximate pressure converges  to the
exact pressure in the $L^2$-norm also with the rate $O(h^{k+\frac12})$. For the BDM space the rate
$O(h^{k+\frac12})$ is obtained for the projection of the error onto
the pressure space, but since in this case the pressure space is of
polynomial degree $k-1$, this is a superconvergence result. 
\subsection{Linear model problem}\label{sec:linear}
To keep the
discussion as simple as possible we consider the following linear
model problem.

Find a velocity $\bu$ and a pressure $p$ satisfying
\begin{subequations}\label{pde}
\begin{alignat}{2}
\dive (\bu \otimes \bbeta) + \sigma \bu+ \nabla p= &\bff \quad  && \text{ in }  \Omega\,, \\
\dive \bu=&0  \quad  && \text{ in } \Omega\,,  \\
\bu \cdot \bn=& 0  \quad && \text{ on } \Gamma.
\end{alignat}
\end{subequations}
We think of $\bu$ and $\bbeta$ as column vectors and we set $\bu \otimes \bbeta = \bu \bbeta^t$. We assume that $\dive \bbeta=0$ and that $\sigma\in L^{\infty}(\Omega)$ with $\sigma(\bx)\ge \sigma_0^{}  > 0$ 
almost everywhere in $\Omega$.  We assume that $\bbeta \cdot \bn=0$ on
$\Gamma$. In spite of it being the natural candidate for a model
problem for the
development and analysis of numerical methods for inviscid flow this
model does not seem to have been considered in the literature. Below
we will first discuss the flow modelling leading to the system
\eqref{pde}. 

To obtain the stationary linear model problem \eqref{pde} from the
incompressible Euler equations, assume that a stationary solution to
the latter $\bbeta$, is subject to a smooth, exponentially
growing perturbation of the right hand side of the momentum equation of the form:
\[
\tilde \ff(\bx,t) := \ff(\bx) \exp (\sigma t), \quad \sigma \in
\mathbb{R}\setminus 0.
\]
Writing the perturbed solution $\bbeta + \tilde \bu$ where $\tilde \bu(x,t)$ is the perturbation
resulting from the pertubation of the right hand side and neglecting
quadratic terms in the perturbation $\tilde \bu$, we may write the
linearised momentum equation 
\begin{equation}\label{eq:momentum}
\partial_t \tilde \bu + \dive (\tilde \bu \otimes \bbeta) +  \dive
(\bbeta \otimes \tilde\bu) +  \nabla \tilde p
=\tilde  \ff(\bx,t).
\end{equation}
With the above choice of perturbation we may write the solution on the
separated form 
\[
\tilde \bu(\bx,t) = \bu(\bx) \exp(\sigma t).
\]
Injecting this expression in \eqref{eq:momentum} we arrive at the
following stationary form for the space varying part of the
perturbation
\begin{equation}\label{eq:stat_momentum}
\sigma \bu + \dive (\bu \otimes \bbeta) + \dive
(\bbeta \otimes \bu) +\nabla p = \ff(\bx).
\end{equation}
To further simplify the model problem we finally drop the second term in the
left hand side of \eqref{eq:stat_momentum}. Since $\dive
(\bbeta \otimes \bu)  = \bu \cdot \nabla \bbeta$, this is a non-essential
term which can be absorbed in the reaction term under suitable assumptions on $\sigma$
and $\bbeta$.

It is easy to construct solutions to the system
\eqref{pde}. Examples of such solutions in the unit square are
\begin{enumerate}
\item x-independent solution.\\
Let $\bbeta \cdot \bn = 0$ on $y=0$ and $y=1$ and $\bbeta$ is defined to be
periodic at $x=0$ and $x=1$. Then for any function $\varphi:\R
\mapsto \R$, $\varphi \in [C^1(\R)]^2$ a solution is given by:
\[
\bbeta := \left(
\begin{array}{c}
\varphi(y) \\
0
\end{array}
 \right).
\]
The associated pressure is $p=0$.
\item Stationary vortex sheet.\\
Let $\bbeta \cdot \bn = 0$ on the boundaries of the square and define the
streamfunction $\varphi
(x,y):= \sin (n \pi x) \sin (n \pi y)$, corresponding to the vorticity
$\omega := \Delta \varphi = - 2 n^2 \pi^2
  \sin (n \pi x) \sin (n \pi y) = -2  n^2 \pi^2 \varphi$ with $n$ a positive integer. Then define:
\begin{equation}\label{eq:vel_vortex}
\bbeta := \left(
\begin{array}{c}
\partial_y \varphi(x,y) \\
-\partial_x \varphi(x,y)
\end{array}
 \right).
\end{equation}
Since $\bbeta \cdot \nabla \omega = -2 n^2 \pi^2 (\partial_y
\varphi(x,y) \partial_x \varphi(x,y)- \partial_x \varphi(x,y) \partial_y
\varphi(x,y)) = 0$ we see that $\bbeta$ is a solution to the
two-dimensional stationary equations of inviscid flow. It is
straightforward to verify that the velocity pressure formulation is
satisfied for the pressure, 
\begin{equation}\label{eq:press_vortex}
p = n^2 \pi^2 (\cos^2 (n \pi x)-\sin^2 (n \pi y))/2.
\end{equation}
\end{enumerate}
In both examples (1) and (2) we achieve a problem on the form \eqref{pde} by taking $\ff = \sigma
\bbeta$ and the solution is then $\bu = \bbeta$.
\subsection{Outline of paper}
We
prove existence of solutions of the model problem \eqref{pde} and uniqueness for $\sigma$
large enough, on smooth domains, in section \ref{sec:wp}.
The H(div)-conforming upwind finite element
methods are introduced and analysed in section \ref{sec:upw}.
Finally in section \ref{sec:num} we illustrate the theory by computing
approximations to the example (2) above. 

\section{Notation and preliminary results}
The partial differential equation will be posed on an
  open polyhedral domain $\Omega\subseteq\mathbb{R}^d, d=2,3$ with
  Lipschitz boundary $\Gamma$. For some of the theoretical results we
  will assume a smoother boundary.
 We adopt standard notation for Sobolev and Lebesgue spaces. In particular, for $D\subset\Omega$ we denote by $(\cdot,\cdot)_D^{}$ the $L^2(D)$ inner product
(without making a distinction between scalar and vector and tensor-valued functions). 
For $D=\Omega$ we drop the subindex in the above notation.
The norm in $L^2(D)$ will be denoted by $\|\cdot\|_D^{}$.  By $W^{m,p}(D), m\ge 0, 1\le p\le \infty$ 
we will denote the functions in $L^p(D)$, with distributional derivatives up to order $m$
belonging to $L^p(D)$, with norm (seminorm) $\|\cdot\|_{m,p,D}^{}$ ($|\cdot|_{m,p,D}^{}$).
For $p=2$ we denote $H^m(D)=W^{m,2}(D)$, and the corresponding norm is
denoted $\|\cdot\|_{m,D}$. 
As usual, $H^m_0(D)$ 
denotes the closure of $C^\infty_0(D)$ in the
$\|\cdot\|_{m,D}^{}$-norm. We also denote by $L^2_0(D)$ the space of
$L^2(D)$ functions with zero mean value in $D$. All spaces for vector-valued
functions will be denoted by boldface notation, e.g.,
$\bH^1(D)=[H^1(D)]^d$, hence we denote by $\bH(\dive,D)$ the space of
$\bL^2(D)$ functions with distributional divergence in $\bL^2(D)$,  $\bH_0^{}(\dive,D)=\{\bv\in \bH(\dive,D):\bv\cdot\bn =0\;\textrm{on}\;\partial D\}$, and $\bH(\curl,D)$ denotes the space
of $\bL^2(D)$ functions with distribution curl in $\bL^2(D)$. 

Below we will make use of the following preliminary result (for its proof, see, e.g., \cite{girault2012finite}).
\begin{lemma}\label{infsup}
There exists a constant $C>0$ such that for every $q \in L_0^2(\Omega)$ there exists $\bv \in \bH_0^1(\Omega)$ satisfying
\begin{equation*}
\dive \bv=q \qquad \text{ in } \Omega,
\end{equation*}
\begin{equation*}
\|\nabla \bv\|_{\Omega} \le C \|q\|_{\Omega}.
\end{equation*}
\end{lemma}

Also in  \cite{girault2012finite} the proof of the following result can be found.
\begin{proposition}
The following bound holds
\begin{equation}\label{l2}
\|\bv\|_{\Om} \le C (\| \dive \bv \|_{\Om} +\| \curl \bv\|_{\Om}) \qquad \forall \bv \in \bH_0(\dive,\Omega) \cap \bH(\curl,\Omega). 
\end{equation}
If we assume that $\partial \Om$ is $C^{1,1}$ 
 \begin{equation}\label{h1}
\|\nabla \bv\|_{\Om} \le K(\| \dive \bv \|_{\Om} +\| \curl \bv\|_{\Om}) \qquad \forall \bv \in  \bH_0(\dive,\Omega) \cap \bH(\curl,\Omega). 
\end{equation}
Finally, if $\Om$ is a convex Lipschitz polyhedron \cite{ABDG98}, or a convex more regular domain, then
\begin{equation}\label{h1convex}
\|\nabla \bv\|_{\Om}^2 \le \| \dive \bv \|_{\Om}^2 +\| \curl \bv\|_{\Om}^2 \qquad \forall \bv \in  \bH_0(\dive,\Omega) \cap \bH(\curl,\Omega). 
\end{equation}
\end{proposition}
 Finally, for two $3 \times 3$ matrices $A$ and $B$ with rows $A_i$ and $B_i$ ($i=1,2,3$) we define
 $C:=A \times B$ with $C_1= A_2 \cdot B_3-A_3 \cdot B_2$, $C_2= -(A_1 \cdot B_3-A_3 \cdot B_1)$ $C_3= A_1 \cdot B_2-A_2 \cdot B_1)$, and
a simple calculation gives the following identity.  
 \begin{lemma}\label{derivative-identity}
 It holds
 \begin{equation*}
 \curl( \bbeta \cdot \nabla \bv)= \bbeta \cdot \nabla (\curl \bv)+ ((\nabla \bbeta)^t \times \nabla \bv) . 
 \end{equation*}
 \end{lemma}


\section{Well-Posedness of the model problem}\label{sec:wp}
It appears that the linear inviscid model \eqref{pde} has not been
analysed mathematically. Hence, will here first study its
well-posedness before proceeding with the finite element analysis.  Transport problems have been studied by several authors (e.g. \cite{girault2010lp,fichera1963unified, da1986stationary}).  However,  the incompressibility constraint seems to add new challenges to the analysis and we cannot apply the techniques of the above mentioned papers directly. 
The weak formulation of \eqref{pde} is given by:

Find $\bu \in \bH_0(\dive,\Omega)$ and $p \in L_0^2(\Omega)$ that satisfy
\begin{subequations}\label{weak1}
\begin{alignat}{2}
-(\bu, \bbeta \cdot \nabla \bv)+(\sigma \bu, \bv)-(p, \dive \bv)=&(\bff, \bv) \quad  &&\text{ for all  } \bv \in \bH^1(\Omega) \cap \bH_0(\dive,\Omega), \label{weak1_1}\\
(\dive \bu, q)=&0 \quad && \text{ for all  } q \in L_0^2(\Omega)  \label{weak1_2}.
\end{alignat}
\end{subequations}

\subsection{Existence of weak solutions \eqref{weak1}}
In order to prove existence of the problem \eqref{weak1} we will regularize it. Consider the following problem:
Find a velocity $\bu_{\eps}$ and a pressure $p_{\eps}$ satisfying
\begin{subequations}\label{pdeeps}
\begin{alignat}{2}
-\eps \Delta \bu_{\eps}+ \dive (\bu_\eps \otimes \bbeta) + \sigma \bu_{\eps}+ \nabla p_{\eps}= &\bff \quad  && \text{ in }  \Omega\,, \\
\dive \bu_{\eps}=&0  \quad  && \text{ in } \Omega\,,  \\
\bu_{\eps}=& 0 \qquad && \text{ on } \Gamma.
\end{alignat}
\end{subequations}
The weak formulation of \eqref{pdeeps} is as follows:
Find $(\bu_{\eps},p_{\eps}) \in \bH_0^1(\Omega)\times L_0^2(\Omega)$ such that
\begin{subequations}\label{weak1eps-Oseen}
\begin{alignat}{2}
\eps (\nabla \bu_{\eps}, \nabla \bv)-(\bu_{\eps}, \bbeta \cdot \nabla \bv)+(\sigma \bu_{\eps}, \bv)-(p_{\eps}, \dive \bv)=&(\bff, \bv) \quad  &&\text{ for all  } \bv \in \bH_0^1(\Omega), \label{weak1eps1-0}\\
(\dive \bu_{\eps}, q)=&0 \quad && \text{ for all  } q \in L_0^2(\Omega)\,.
\end{alignat}
\end{subequations}

\begin{lemma}\label{Lem2.1}
There exists a unique solution $\bu_{\eps} \in \bH_0^1(\Omega)$ and $p_{\eps} \in L_0^2(\Omega)$ to the problem \eqref{weak1eps-Oseen}.
In addition, if $\bbeta\in \bL^\infty(\Omega)$, then the following bound holds
\begin{equation}\label{bound-Oseen}
\|p_{\eps}\|_{\Omega}+  \|\sqrt{\sigma}\,\bu_{\eps}\|_{\Omega}  +\sqrt{\eps} \| \nabla \bu_{\eps}\|_{\Omega} \le C\, \|\bff\|_{\Omega}\,,
\end{equation}
where the constant $C$ depends on $\sigma$ and $\|\bbeta\|_{\infty,\Omega}^{}$, but not on negative powers of $\varepsilon$.
 \end{lemma}

\begin{proof}
Existence and uniquness of a solution of \eqref{weak1eps-Oseen} follows from  the Babuska-Brezzi theory \cite{BBF13}. Testing the equation with $\bu_{\eps}$ we get
\begin{equation}
\eps \|\nabla \bu_{\eps}\|_{\Omega}^2 +\|\sqrt{\sigma}\,\bu_{\eps}\|_{\Omega}^2 =(\bff, \bu_{\eps}).
\end{equation}
Therefore, we have the bound
\begin{equation}\label{bound-1}
\eps \|\nabla \bu_{\eps}\|_{\Omega}^2 + \frac{1}{2} \|\sqrt{\sigma}\,\bu_{\eps}\|_{\Omega}^2 \le \frac{1}{2\sigma_0^{}} \|\bff\|_{\Omega}^2.
\end{equation}
Moreover, using  Lemma \ref{infsup} and \eqref{weak1eps1-0} we have that
\begin{equation*}
\|p_{\eps}\|_{\Omega} \le C \,\Big(\eps \|\nabla \bu_{\eps}\|_{\Omega}+ \|\bbeta\|_{\infty,\Omega}^{} \|\bu_{\eps}\|_{\Omega} 
+ \|\sqrt{\sigma}\|_{\infty,\Omega}\,\|\sqrt{\sigma}\,\bu_{\eps}\|_\Omega+\|\bff\|_\Omega
\Big)\,,
\end{equation*}
and the proof is finished using \eqref{bound-1}.
\end{proof}

\begin{theorem}
There exists a solution $\bu \in \bL^2(\Omega)$ and $p \in L^2(\Omega)$ to \eqref{weak1}.
\end{theorem}

\begin{proof}
Since  $\{\bu_{\eps}\}$  and $\{p_{\eps}\}$ are uniformly bounded in $\bH_0^{}(\dive,\Omega)$ and $L^2_0(\Omega)$, respectively,
there exists a subsequence such that 
$\bu_{\eps} \rightharpoonup \bu$ and $p_{\eps} \rightharpoonup p$ with $\bu \in \bH_0^{}(\dive,\Omega)$ and $p \in L_0^2(\Omega)$. 
Moreover, since $\dive \bu_\eps=0$, for all $\phi \in H_0^1(\Omega)$ we have $(\bu, \nabla \phi) =\lim_{\eps \rightarrow 0} (\bu_{\eps}, \nabla \phi) =\lim_{\eps \rightarrow 0}  -(\dive \bu_{\eps}, \phi) =0$, thus showing that $\dive\bu=0$ in $\Omega$.
 We then see that from \eqref{weak1eps1-0}  and the fact that $\eps \|\nabla \bu_{\eps}\|_{\Omega} \rightarrow 0 $ as $\eps\to 0$ that $\bu$ and $p$  satisfy \eqref{weak1}.
\end{proof}

\subsection{Uniqueness of weak solutions}
In general we cannot prove uniqueness of weak solutions \eqref{weak1}. However, we will be to prove existence and uniqueness of solutions in the space $\bH^1(\Omega) \cap \bH_0(\dive,\Omega)$ by making more stringent requirements on the coefficients and the boundary $\Gamma$. 
 To achieve this goal, it is necessary to introduce a different regularised (as compared to \eqref{pdeeps}) problem to prove existence of smoother solutions to \eqref{weak1}. The idea
consists in considering the folllowing regularised Hodge-Oseen problem:  Find a velocity $\bu_{\eps}$ and a pressure $p_{\eps}$ satisfying
\begin{subequations}\label{pdeeps2}
\begin{alignat}{2}
\eps \, \curl \curl \bu_{\eps}+ \dive (\bu_\eps \otimes \bbeta) + \sigma \bu_{\eps}+ \nabla p_{\eps}= &\bff \quad  && \text{ in }  \Omega\,, \\
\dive \bu_{\eps}=&0  \quad  && \text{ in } \Omega\,,  \\
\bu_{\eps}\cdot \bn =& 0 \qquad && \text{ on } \Gamma, \\
\curl \bu_{\eps} \times \bn=& 0 \qquad && \text{ on } \Gamma.
\end{alignat}
\end{subequations}
The weak formulation of \eqref{pdeeps2} reads as follows:
Find $\bu_{\eps} \in \bV:= \bH^1(\Omega) \cap \bH_0(\dive,\Omega)$ and $p_{\eps} \in L_0^2(\Omega)$ that satisfy
\begin{subequations}\label{weak1eps}
\begin{alignat}{2}
\eps (\curl \bu_{\eps}, \curl \bv)-(\bu_{\eps}, \bbeta \cdot \nabla \bv)+(\sigma \bu_{\eps}, \bv)-(p_{\eps}, \dive \bv)=&(\bff, \bv) \quad  &&\text{ for all  } \bv \in \bV, \label{weak1eps1}\\
(\dive \bu_{\eps}, q)=&0 \quad && \text{ for all  } q \in L_0^2(\Omega).
\end{alignat}
\end{subequations}

\begin{theorem}\label{regcurl}
Assume that $\bff \in \bL^2(\Om)$ and that $\Gamma$ is $C^{1,1}$, or $\Omega$ is a
convex Lipschitz polyhedron. Then, there exists a unique 
 solution of \eqref{weak1eps}. In addition, it satisfies
\begin{equation}\label{213}
\sqrt{\eps} \|\curl \bu_{\eps}\|_{\Om}+ \|\sqrt{\sigma}\,\bu_{\eps}\|_{\Om}+ \|p_{\eps}\|_{\Om} \le C \|\bff\|_{\Omega}.
\end{equation}
Moreover, suppose that $\bff \in \bH^1(\Om), \bbeta\in
\boldsymbol{W}^{1,\infty}(\Omega)$, $\sigma \in W^{1,\infty}(\Omega)$ and  $\Gamma$ is $C^{3}$.  If $\Omega$ is convex, let  $\mathcal{C}= \|\nabla \bbeta\|_{L^\infty(\Omega)}$, or otherwise $\mathcal{C}= K \|\nabla \bbeta\|_{\infty,\Omega}$  where $K$ is from \eqref{h1}. Then, assuming $\sigma_0^{} > \mathcal{C}$ we have
\begin{equation*}
\| \curl \bu_{\eps}\|_{\Om} \le  C\,\| \bff\|_{\curl,\Om}\,,
\end{equation*} 
where $C>0$ depends on $\sigma, \bbeta$, and $K$, but not on negative powers
of $\eps$.
\end{theorem}

\begin{proof}
The existence and uniqueness of this solution follows from the Babuska-Brezzi theory \cite{BBF13} by noting 
that as proven in \cite{girault2012finite}, the norm in $\bH^1(\Omega)$ is equivalent to the one in 
$\bH(\curl,\Omega) \cap \bH_0(\dive,\Omega)$, thanks to the hypotheses on $\Gamma$.  The bound \eqref{213}  follows taking $\bv=\bu_\eps$ in \eqref{weak1eps1}, and  the inf-sup conditions
provides the stability for $p_{\eps}$. 

Next, whenever we suppose that $\Gamma$ is of class $C^3$ and $\bff\in \bH^1(\Omega)$,
using the results in \cite{sil2017regularity} (see Thereom 12 and
Remark 16) we have the regularity $\bu_{\eps} \in \bH^3(\Omega)$ and
$p_{\eps} \in H^2(\Omega)$. Noting that $\curl \bu_{\eps} \times \bn =
0$ on $\Gamma$ it follows that $\curl \curl (\bu_{\eps}) \cdot \bn = 0$, so, 
$\tilde{\bv}:= \curl \curl (\bu_{\eps}) \in \bH^1(\Omega) \cap \bH_0(\dive,\Omega)$, and then it is a valid test function
to be used in \eqref{weak1eps}. Thus, taking $\tilde{\bv}$ as test function in \eqref{weak1eps} and integrating
by parts we obtain
\begin{equation}\label{uniq-1}
\eps \|\curl \curl \bu_{\eps}\|_{\Om}^2 -(\bu_{\eps}, \bbeta \cdot \nabla (\curl \curl \bu_{\eps}))+ 
 \|\sqrt{\sigma}\,\curl \bu_{\eps}\|_{\Om}^2  + (\nabla\sigma\times \bu_\eps, \curl\bu_\eps) =(\curl \bff, \curl \bu_{\eps})\,.
\end{equation}
The second term in the left can be written as 
\begin{equation*}
 -(\bu_{\eps}, \bbeta \cdot \nabla (\curl \curl \bu_{\eps}))=(\bbeta \cdot \nabla\bu_{\eps}, \curl \curl \bu_{\eps})=(\curl (\bbeta \cdot \nabla \bu_{\eps}), \curl \bu_{\eps})\,.
\end{equation*}
However,  using Lemma~\ref{derivative-identity} and the antisymmetry of the convective term
\begin{equation*}
(\curl (\bbeta \cdot \nabla \bu_{\eps}), \curl \bu_{\eps})= (\bbeta \cdot \nabla (\curl \bu_{\eps}), \curl \bu_{\eps})+ ((\nabla \bbeta)^t \times \nabla  \bu_{\eps}, \curl \bu_{\eps})= ((\nabla \bbeta)^t \times \nabla  \bu_{\eps}, \curl \bu_{\eps})\,,
\end{equation*}
and then
\begin{equation*}
-(\bu_{\eps}, \bbeta \cdot \nabla (\curl \curl \bu_{\eps}))=((\nabla \bbeta)^t \times \nabla  \bu_{\eps}, \curl \bu_{\eps}).
\end{equation*}
Therefore, replacing the last identity in \eqref{uniq-1} we have 
\begin{equation*}
\eps \|\curl \curl \bu_{\eps}\|_{\Om}^2 + \|\sqrt{\sigma}\,\curl \bu_{\eps}\|_{\Om}^2=
(\curl \bff-\nabla\sigma\times \bu_\eps, \curl\bu_\eps)-((\nabla \bbeta)^t \times \nabla  \bu_{\eps}, \curl \bu_{\eps}).
\end{equation*}
Using the Cauchy Schwarz inequality, one of the inequalities \eqref{h1} or \eqref{h1convex}, and the fact that $\dive \bu_{\eps}=0$ we have
\begin{equation*}
\|\sqrt{\sigma}\,\curl \bu_{\eps}\|_{\Om}^2 \le  \| \curl \bff-\nabla\sigma\times \bu_\eps \|_{\Omega}\|\curl \bu_{\eps}\|_{\Omega} +  \mathcal{C}  \|\curl \bu_{\eps}\|_{\Omega}^2\,.
\end{equation*}
Hence,
\begin{equation*}
(\sigma_0- \mathcal{C})\,\|\curl \bu_{\eps}\|_{\Om}^2  \le  
 \|\curl \bff-\nabla\sigma\times \bu_\eps \|_{\Omega} \|\curl \bu_{\eps}\|_{\Omega}\,,
\end{equation*}
and the proof follows dividing by $\|\curl \bu_{\eps}\|_{\Om}$ and applying \eqref{213}.
\end{proof}

\begin{theorem}\label{uniqueness}
Let us assume the hypotheses Theorem \ref{regcurl} . Then, there exists a unique 
solution of \eqref{weak1} such that $\bu \in \bH^1(\Om) \cap \bH_0(\dive,\Omega)$ and $p \in H^1(\Omega)$.
\end{theorem}

\begin{proof}
Let $(\bu_{\eps}, p_{\eps})$ be the solution of \eqref{weak1eps}. Then, by Theorem~\ref{regcurl}  $\{(\bu_{\eps},p_{\eps})\}$ is uniformly bounded in $\bH^1(\Om)\times L^2_0(\Omega)$.  Hence, there exists 
a subsequence such that $\bu_{\eps}  \rightharpoonup \bu \in \bH^1(\Omega)$ and
$p_{\eps}  \rightharpoonup p \in L_0^2(\Omega)$ weakly. Moreover, since $\dive \bu_{\eps}=0$ and 
$\bu_{\eps} \in \bH_0(\dive,\Omega)$, then $\dive \bu=0$ and $\bu \in \bH_0(\dive,\Omega)$.
This proves that $\bu$ satisfies the second equation in \eqref{weak1} and the boundary conditions.
Since $\bu_{\eps}  \rightharpoonup \bu \in \bH^1(\Omega)$ weakly, then $\bu_{\eps}  \to \bu$
strongly in $\bL^2(\Omega)$.  In addition, since $\eps \|\curl \bu_{\eps}\|_{\Om}\to 0$ as
$\eps\to 0$, then using the weak convergence of $p_\eps$ to $p$
in $L^2(\Omega)$ we can take the limit as $\eps\to 0$ in \eqref{weak1eps} and conclude that
$(\bu, p)$ also satisfies the first equation in \eqref{weak1}. Finally, from the first equation in \eqref{weak1} we have
$\nabla p =\bff -\sigma\,\bu - \dive(\bbeta\otimes\bu)\in \bL^2(\Omega)$, and then $p\in H^1(\Omega)$.

To prove uniqueness, assume that $\bff=\boldsymbol{0}$. If we test with  $\bu \in \bH^1(\Omega) \cap \bH_0(\dive,\Omega)$ we immediately get that $\|\sqrt{\sigma}\bu\|_{\Om}^2=0$ which gives that $\bu=\boldsymbol{0}$. It easily follows that $p=0$.
\end{proof}

We finish this section by stating the following result that, in essence, casts the problem \eqref{weak1} as the limit of the Oseen problem \eqref{weak1eps-Oseen}.
\begin{corollary} Under the same hypotheses from Theorem \ref{uniqueness} the solution $(\bu,p)$ of \eqref{weak1} is the limit of the solutions of the Oseen problem \eqref{weak1eps-Oseen} in the following sense
\begin{equation}\label{limit}
\lim_{\eps\to 0}\Big( \|\bu-\bu_\eps\|_{\dive,\Omega}+\|p_\eps-p\|_\Omega\Big) = 0\,.
\end{equation}
\end{corollary}
\begin{proof}
The error $(\bu-\bu_\eps,p-p_\eps)$ satisfies the following error equation
\begin{equation}\label{error-1}
(\sigma(\bu-\bu_\eps),\bv)-\eps(\nabla\bu_\eps,\nabla\bv)+(\bbeta\otimes(\bu-\bu_\eps),\nabla\bv) -(p-p_\eps,\dive\bv) = 0\,,
\end{equation}
for all $\bv\in \bH^1(\Omega)\cap \bH_0(\dive,\Omega)$. Since $\bu\in \bH^1(\Omega)\cap \bH_0(\dive,\Omega)$,  $\hat{\bv}:=\bu-\bu_\eps$ is a valid test function for \eqref{error-1}. So, using $\hat{\bv}$
in \eqref{error-1}, 
the fact that both $\bu$ and $\bu_\eps$ are divergence-free, 
the Cauchy-Schwarz inequality, and \eqref{bound-Oseen} we get
\begin{align}
\|\sqrt{\sigma}(\bu-\bu_\eps)\|^2_\Omega +\eps\|\nabla(\bu-\bu_\eps)\|^2_\Omega &= \eps(\nabla\bu,\nabla(\bu-\bu_\eps)) \nonumber\\
&\le 
\sqrt{\eps}\|\nabla\bu\|_\Omega\,\sqrt{\eps}\|\nabla(\bu-\bu_\eps)\|_\Omega\to 0\,,
\end{align}
as $\eps\to 0$, which proves the convergence of $\bu_\eps$ to $\bu$ in $\bL^2(\Omega)$. The convergence of $\bu_\eps$ to $\bu$ in $\bH_0(\dive,\Omega)$
follows from the fact that both $\bu_\eps$ and $\bu$ are divergence-free.
\\ { } \\
To prove the convergence of the pressure, using Lemma~\ref{infsup} there exists $\bw\in \bH^1_0(\Omega)$ such that $|\bw|_{1,\Omega}\le C\|p-p_\eps\|_\Omega$ and $\dive\bw =p-p_\eps$.
Then, using \eqref{error-1}, the Cauchy-Schwarz inequality, and the convergence of $\bu_\eps$ to $\bu$,
\begin{align}
\|p-p_\eps\|^2_\Omega &= (p-p_\eps,\dive\bw) = (\sigma(\bu-\bu_\eps),\bw)+\eps(\nabla\bu_\eps,\nabla \bw) + (\bbeta\otimes (\bu_\eps-\bu),\nabla\bw) \nonumber\\
&\le C\,\Big(\|\sqrt{\sigma}\|_{\infty,\Omega}\,\|\sqrt{\sigma}\,(\bu-\bu_\eps)\|_\Omega+\sqrt{\eps}\,\sqrt{\eps}\|\nabla\bu_\eps\|_\Omega+\|\bbeta\|_{\infty,\Omega}\|\bu-\bu_\eps\|_\Omega\Big)\,\|p-p_\eps\|_\Omega\,,
\end{align}
and the proof follows by dividing by $\|p-p_\eps\|_\Omega$ and noticing that, thanks to
\eqref{bound-Oseen} the term within parentheses tends to zero as $\eps\to 0$.
\end{proof}

\section{Upwind H(div) method}\label{sec:upw}
\subsection{Preliminaries}
We denote by $\{ \mct\}_{h>0}^{}$ a family of shape-regular simplicial  triangulations of 
$\Omega$. The elements of $\mct$ are denoted by $T$, with diameter $h_T$, and $h:=\max\{h_T:T\in\mct\}$.
The set of its facets (edges for $d=2$, faces for $d=3$) is denoted by $\mathcal{E}_h$.
To cater for the nonconforming
character of the approximation we also introduce the following broken
versions of the scalar product
\begin{alignat*}{1}
(\bv, \bw)_h=&\sum_{T \in \mct} \int_T \bv \cdot \bw \, dx, \\
\langle \bv, \bw \rangle_h=&\sum_{T \in \mct} \int_{\partial T} \bv \cdot \bw \, ds.  
\end{alignat*}
In addition, we introduce the broken space $H(\mct)$, of functions in $L^2(\Omega)$ whose restriction to every $T\in\mct$ belongs to $H(T)$.

Let $T \in \mct$ and let $\bx \in \partial T$ then we define
\begin{equation*}
\bv_{\bbeta}^\pm(\bx)=\lim_{\epsilon \rightarrow 0} \bv\big(\bx \pm \epsilon (\bbeta(\bx) \cdot \bn(\bx)) \bn(\bx)\big). 
\end{equation*}
and
\begin{equation*}
\hat{\bv}(\bx)= \, \bv_{\bbeta}^-(\bx) 
\end{equation*}
For $F \in \mathcal{E}_h$ and $F= \partial T_1 \cap \partial T_2$ for $T_1, T_2, \in \mct$ we define the jumps
\begin{equation*}
\jump{\bv \otimes \bn}|_{F}= \bv|_{T_1} \otimes \bn_1+ \bv|_{T_2} \otimes \bn_2\,,
\end{equation*}
and for $F \in \mathcal{E}_h$ and $F \subset \Gamma$ we define 
\begin{equation*}
\jump{\bv \otimes \bn}|_{F}= \bv \otimes \bn.
\end{equation*}

We then define the semi-norm on the jumps of the solution over element
boundaries to be
\begin{equation*}
|\bv|_{\bbeta}^2=\sum_{F \in \mathcal{E}_h} \|\sqrt{|\bbeta \cdot \bn|} \jump{\bv \otimes \bn}\|_{0,F}^2.
\end{equation*}
With these definitions we can state the following important identity
\cite[Lemma 6.1]{GRW05}
\begin{proposition}\label{proposition1}
For all $\bv\in \bH^1(\mct)$, the following holds
\begin{equation}\label{aux1}
(\bv \otimes \bbeta, \nabla \bv)_h-\langle \bbeta \cdot \bn \hat{\bv}, \bv \rangle_h= -\frac{1}{2}  |\bv|_{\bbeta}^2.
\end{equation}
\end{proposition}

Let us define the Raviart-Thomas \cite{RT} and BDM spaces \cite{BDM}. The space of
polynomials of degree at most $k$ defined in $T$ is denoted by $\pol_k(T)$, and we denote
$\bpol_k(T)=[\pol_k(T)]^d$. For every $T\in \mct$, let $\text{RT}_k(T)= \bpol_k(T)+
(\pol_k(T) \setminus \pol_{k-1}(T)) \bx$. We define, for $k\ge 0$, the spaces
\begin{alignat*}{1}
\bV_{h,k}^{\text{RT}}=&\{ \bv \in \bH_0(\dive , \Omega):  \bv|_T \in \text{RT}_k(T) \text{ for all } T \in \mct \}, \\
\bV_{h,k}^{\text{BDM}}=&\{ \bv \in \bH_0(\dive ,  \Omega):  \bv|_T \in \bpol_k(T) \text{ for all } T \in \mct\}, \\
M_{h,k}=&\{ q \in L_0^2(\Omega):  q|_T \in \pol_k(T) \text{ for all } T \in \mct\}.
\end{alignat*}

A well-known property linking these two spaces is stated now (for a
proof see \cite[Lemma 4.3]{CG04}).

\begin{lemma}\label{auxlemma}
Let  $\bv \in    \bV_{h,k}^{\text{RT}}$ with $\dive \bv=0$ on $\Omega$ then  $\bv \in    \bV_{h,k}^{\text{BDM}}$.
\end{lemma}

We next introduce the standard
$L^2$-projection on polynomials on an element $T$, $P_k^T:L^2(T)
\rightarrow \pol_{k}(T)$. Its global equivalent will be denoted $P_k:L^2(\Omega)
\rightarrow M_{h,k}$. We recall the standard estimates for the
$L^2$-projection (see, e.g., \cite{EG04})
\begin{align}
\|P_k q - q \|_\Omega+ h \|\nabla (P_k q - q )\|_\Omega \le C \, h_T^{k+1} |q|_{k+1,\Omega}\,, \label{eq:L2approx}\\
\|q -P_0 q\|_{\infty,T} \le C \, h_T \| q\|_{1,\infty,T}\,.\label{eq:Linftyapprox}
\end{align}

The Raviart-Thomas interpolation operator will be used in the sequel. It is defined as follows: $\bPi:  \bH^1(\Omega) \cap \bH_0(\text{div}; \Omega) \rightarrow \bV_{h,k}^{\text{RT}}$ 
where $\bPi \bv$ is the only function of $\bV_{h,k}^{RT}$ satisfying
\begin{align}
\int_T (\bPi \bv- \bv) \cdot \bw\,dx=&\,0 \quad \text{ for all }  \bw \in \bpol_{k-1}(T), \textrm{and all}\; T \in \mct,  \label{RT-1a}\\
\int_F (\bPi \bv -\bv) \cdot \bn w \, ds=&\,0 \quad \text{ for all } w \in \pol_{k}(F),  \textrm{and all}\;  F \in \mathcal{E}_h.\label{RT-1b}
\end{align}
This operator satisfies the following classical properties (see, e.g., \cite{BBF13}).
\begin{lemma}\label{errorproj}
Let $k\ge 0$. The mapping $\Pi$ satisfies the following commutative property
\begin{equation}\label{commutative}
\dive \bPi \bv= P_k \dive \bv\,.
\end{equation}
Let $\bv \in \bH^{k+1}(\Omega)$ then we have
\begin{equation*}
\|\bPi \bv-\bv\|_{T}+ h_T \|\nabla (\bPi \bv-\bv)\|_{T} \le C \, h_T^{k+1} |\bv|_{k+1,T} \quad \text {for all } T \in \mct. 
\end{equation*}
\end{lemma}

We end this section recalling the following classical inverse and local trace inequalities that hold for every $T\in\mct$
\begin{align}
|v_h^{}|_{1,T}^{} &\le Ch^{-1}\|v_h^{}\|_{T}^{}\qquad\forall\, v_h^{}\in \pol_k(T)\,, \label{inverse}\\
\|v\|_{\partial T} &\le C\big( h_T^{-\frac{1}{2}}\|v\|_{T}+h_T^{\frac{1}{2}}|v|_{1,T}\big)\qquad\forall\, v\in H^1(T)\,. \label{local-trace}
\end{align}

\subsection{The finite element method and the error estimates for the velocity}

Throughout, the velocity and pressure will be approximated using the spaces $\bV_h$ and $M_h$, respectively. In this work we will consider the following
choices:
\begin{equation*}
 \bV_h=\bV_{h,k}^{\text{RT}} \quad \text{ and } M_h=M_{h,k},  \text{ for  } k \ge 0,
\end{equation*}
or 
\begin{equation*}
 \bV_h=\bV_{h,k}^{\text{BDM}} \quad \text{ and } M_h=M_{h,k-1}, \text{ for } k \ge 1.
\end{equation*}

The numerical method analysed here reads: Find $\bu \in \bV_h$ and $p_h \in M_h$ such that
\begin{subequations}\label{fem}
\begin{alignat}{2}
-(\bu_h, \bbeta \cdot \nabla \bv_h)_h+ \langle (\bbeta \cdot \bn) \widehat{\bu_h}, \bv_h  \rangle_h+(\sigma \bu_h, \bv_h) -(p_h, \dive \bv_h )=&\, (\bff,\bv_h) \quad && \text{ for all } \bv_h \in \bV_h,\\
 (\dive \bu_h, q_h)=&\, 0 \quad && \text{ for all } q_h \in M_h. 
\end{alignat}
\end{subequations}
Thanks to the inf-sup stability of the pair $\bV_h\times M_h$ (see \cite{BBF13}), and Proposition~\ref{proposition1}, problem \eqref{fem} has a unique solution. Moreover, the
method~\eqref{fem}  is consistent; in fact,  for $(\bu,p) \in \bH^1(\Omega) \times L^2_0(\Omega)$ solving \eqref{pde} we have
\begin{subequations}\label{weak}
\begin{alignat}{2}
-(\bu, \bbeta \cdot \nabla \bv_h)_h+ \langle (\bbeta \cdot \bn) \bu, \bv_h  \rangle_h+(\sigma \bu, \bv_h) -(p, \dive \bv_h )=&\, (\bff,\bv_h) \quad && \text{ for all } \bv_h \in \bV_h,\\
 (\dive \bu, q_h)=&\, 0 \quad && \text{ for all } q_h \in M_h. 
\end{alignat}
\end{subequations}
A consequence of Lemma \ref{auxlemma} is that the finite element method
\eqref{fem} produces the same velocity approximation for  $\bu_h \in
\bV_{h,k}^{\text{RT}}$ and  $\bu_h \in    \bV_{h,k}^{\text{BDM}}$. We
show that in the following proposition.
\begin{proposition}\label{eq:RTeqBDM}
Let $(\bu_h,p_h)$ be the solution of \eqref{fem}
for the spaces $\bV_h\times M_h =\bV_{h,k}^{\text{RT}} \times M_{h,k}$
and $(\tilde \bu_h, \tilde p_h)$ the solution of \eqref{fem}
for the spaces $\bV_h\times M_h =\bV_{h,k}^{\text{BDM}} \times
M_{h,k-1}$. Then $\bu_h=\tilde \bu_h $.
\end{proposition}
\begin{proof}
Let $\be_h := \tilde \bu_h - \bu_h$, $\eta_h = \tilde p_h - p_h$ then using \eqref{fem} we see that
\begin{equation}\label{eq:peRTBDM}
-(\be_h, \bbeta \cdot \nabla \bv_h)_h+ \langle (\bbeta \cdot \bn) \widehat{\be_h}, \bv_h  \rangle_h+(\sigma \be_h, \bv_h) -(\eta_h, \dive \bv_h )=0 \quad \text{ for all } \bv_h \in \bV_{h,k}^{\text{BDM}}.
\end{equation}
Since $\dive \be_h = 0$ by Lemma \ref{auxlemma}there holds $\be_h \in
\bV_{h,k}^{\text{BDM}}$, which is a valid test function. Taking $\bv_h =
\be_h$ in \eqref{eq:peRTBDM} and applying Proposition
\ref{proposition1} we obtain
\[
\|\sqrt{\sigma} \be_h\|_\Omega = 0,
\]
which proves the claim.
\end{proof}
We can now derive an error estimate for the velocity.  We let
$\be_h=\bPi \bu-\bu_h$ and start by noticing that
\begin{equation}\label{divfree}
\dive \be_h=0.
\end{equation}
Hence, by Lemma \ref{auxlemma} we have $\be_h \in \bV_{h,k}^{\text{BDM}}$ and in particular 
\begin{equation}\label{aux2}
\nabla \be_h|_{T} \in [{\pol}_{k-1}(T)]^{d \times d} \qquad \text{  for all }  T \in \mct.
 \end{equation}

\begin{theorem}\label{thm:main}
Let $\bu \in [H^1(\Omega)]^d$ solve \eqref{pde} and let $\bu_h \in \bV_h$ solve \eqref{fem}.
Then, the following error estimate holds
\begin{alignat*}{1}
\|\sqrt{\sigma} (\bu-\bu_h)\|_\Omega &+ |\bu-\bu_h|_{\bbeta} \le\,    C \left(1+\frac{\| \bbeta\|_{1,\infty,T}}{\sigma_0}\right)  \|\sqrt{\sigma} (\bu-\bPi \bu)\|_\Omega \\
& + C \|\bbeta\|_{\infty,\Omega}^{1/2} \left( \sum_{T\in \mct} \left(\frac{1}{h_T} \| \bu-\bPi \bu\|_{T}^2+ h_T \| \nabla(\bu-\bPi \bu)\|_{T}^2\right) \right)^{\frac{1}{2}}\,,
\end{alignat*}
where the constant $C$ does not depend on $h$, or any physical parameter of the equation.
\end{theorem}

\begin{proof} Using  \eqref{fem}, \eqref{weak}, \eqref{divfree}, and \eqref{aux1} we get
\begin{align*}
\|\sqrt{\sigma}\be_h\|_\Omega^2=&(\sigma (\bu-\bu_h), \be_h)+(\sigma (\bPi \bu-\bu), \be_h) \\
=& ((\bu-\bu_h),  \bbeta \cdot \nabla  \be_h)_h- \langle \bbeta \cdot \bn (\bu-\widehat{\bu_h}) ,   \be_h \rangle_h  +(\sigma (\bPi \bu-\bu), \be_h)   \\
=& ((\bPi \bu-\bu_h),  \bbeta \cdot \nabla \be_h)_h- \langle \bbeta \cdot \bn (\widehat{\bPi \bu}-\widehat{\bu_h}) ,  \be_h \rangle_h  \\
&+((\bu-\bPi \bu),  \bbeta \cdot \nabla \be_h)_h- \langle \bbeta \cdot \bn (\bu-\widehat{\bPi \bu}),  \be_h\rangle_h  +(\sigma (\bPi \bu-\bu), \be_h) \\
=&-\frac{1}{2} |\be_h|_{\bbeta}^2  +((\bu-\bPi \bu),  \bbeta \cdot \nabla \be_h)_h- \langle \bbeta \cdot \bn (\bu-\widehat{\bPi \bu}),  \be_h \rangle_h  +(\sigma (\bPi \bu-\bu), \be_h)\,.
\end{align*}
Hence, we have 
\begin{align}
&\|\sqrt{\sigma}\be_h \|_\Omega^2+\frac{1}{2} |\be_h|_{\bbeta}^2 \nonumber \\
&=(\bu-\bPi \bu,  \bbeta \cdot \nabla \be_h)_h- \langle \bbeta \cdot \bn (\bu-\widehat{\bPi \bu}),  \be_h \rangle_h +(\sigma (\bPi \bu-\bu), \be_h )\,. \label{uno}
\end{align}
We bound each term separately.  
Using \eqref{aux2}, the definition of $\bPi$ \eqref{RT-1a}-\eqref{RT-1b},   \eqref{eq:Linftyapprox}, and \eqref{inverse}, we have
\begin{equation} \label{dos}
(\bu-\bPi \bu,  \bbeta \cdot \nabla \be_h)_h= (\bu-\bPi \bu,  (\bbeta-P_0\bbeta) \cdot \nabla \be_h)_h \le C \| \bbeta\|_{1,\infty,\Omega} \|\bu-\bPi \bu\|_\Omega  \|\be_h\|_\Omega\,.
\end{equation}
Using the contributions from neighbouring elements on the face to express the discrete error on the faces in
terms of jumps, the normal continuity of $\bu$ and $\bPi\bu$, and using the local trace inequality \eqref{local-trace} it is easy to show that
\begin{equation} \label{tres}
- \langle \bbeta \cdot \bn (\bu-\widehat{\bPi \bu}),  \be_h \rangle_h  \le C \,\|\bbeta\|_{\infty,\Omega}^{\frac{1}{2}} |\be_h|_{\bbeta} \left( \sum_{T\in \mct} \left(\frac{1}{h_T} \| \bu-\bPi \bu\|_{T}^2+ h_T \| \nabla(\bu-\bPi \bu)\|_{T}^2\right) \right)^{\frac{1}{2}}.
\end{equation}
Finally, 
\begin{equation} \label{cuatro}
(\sigma (\bPi \bu-\bu), \be_h) \le \| \sqrt{\sigma}   (\bPi \bu-\bu)\|_\Omega \| \sqrt{\sigma} \be_h\|_\Omega.
\end{equation}
Therefore, inserting \eqref{dos}-\eqref{cuatro} into \eqref{uno} we arrive at
\begin{alignat*}{1}
\|\sqrt{\sigma} \be_h\|_\Omega+ |\be_h|_{\bbeta} \le &   C \,\Big(1+\frac{\| \bbeta\|_{1,\infty,\Omega}}{\sigma_0}\Big)  \|\sqrt{\sigma} (\bu-\bPi \bu)\|_\Omega \\
& + C \|\bbeta\|_{\infty,\Omega}^{\frac{1}{2}} \left( \sum_{T\in \mct} \left(\frac{1}{h_T} \| \bu-\bPi \bu\|_{T}^2+ h_T \| \nabla(\bu-\bPi \bu)\|_{T}^2\right) \right)^{\frac{1}{2}}.
\end{alignat*}
The result follows after applying the triangle inequality. 
\end{proof}

The following result appears as a corollary of the last theorem and  Lemma~\ref{errorproj}.

\begin{corollary}\label{cor:conv}
Let $\bu \in [H^{k+1}(\Omega)]^d$ solve \eqref{pde} and let $\bu_h \in \bV_h$ solve \eqref{fem}. Then, the following error estimate holds
\begin{alignat*}{1}
\|\sqrt{\sigma} (\bu-\bu_h)\|_\Omega+ |\bu-\bu_h|_{\bbeta} \le &   C \left(\left[1+\frac{\| \bbeta\|_{1,\infty, T}}{\sigma_0}\right] \|\sqrt{\sigma}\|_{\infty,\Omega} h^{\frac{1}{2}}+\|\bbeta\|_{\infty,\Omega}^{\frac{1}{2}}\right) h^{k+\frac{1}{2}} \|\bu\|_{k+1,\Omega)}.
\end{alignat*}
\end{corollary}

\begin{remark}
The arguments of Theorem \ref{thm:main} and Corollary \ref{cor:conv}
may be used to improve the order obtained Theorem 2.2 of \cite{GSS17}
to $O(h^{k+\frac12})$, if an upwind flux is used. Following the ideas
above, use integration by parts in 
the first term of $I_1$ in the equation after (2.12). Then add and
subtract the exact solution to the approximate solution in term $I_3$
and recombine terms, so that one may use continuity on the norm
augmented with $L^2$-control on the faces the jumps of the approximate
velocity.
\end{remark}

\subsection{$L^2$-error estimates for the pressure approximation}
Since the pressure space is of polynomial degree $k$ for the method
using the $RT$ space for velocity approximation and $k-1$ for the
method using the $BDM$ space, the optimal order that can be obtained
for the error of the pressure approximation in the $L^2$-norm is
$O(h^{k+1})$ and $O(h^k)$, respectively. Here we will prove the
following orders for the pressure error :
\begin{enumerate}
\item in the first case (RT), $O(h^{k+\frac12})$; this is,
 the same suboptimality of $O(h^{\frac12})$ as for the velocity
approximation.
\item in the second case (BDM) we get the optimal convergence $O(h^k)$; considering
that the pressure space is of degree $k-1$. For the discrete error,
i.e. the projection
  of the error on the space $M_h$, we get an $O(h^{k+\frac12})$
  estimate, this is a superconvergence of 
$O(h^{\frac12})$ compared with the approximation property of the space
of constant functions.
\end{enumerate}

\begin{theorem}\label{thm:pres}
Let $(\bu,p) \in \bH^1(\Omega)\times L^2_0(\Omega)$ solve
\eqref{pde} and let $(\bu_h,p_h) \in \bV_h \times M_h$ solve
\eqref{fem}. Let $\ell$ denote the polynomial degree of the space $M_h$. Then, the following error estimate holds
\begin{alignat*}{1}
\|P_\ell p - p_h\|_\Omega \leq\, &   C (\|\bbeta\|_{\infty,\Omega} \sigma_0^{-\frac12} + \sigma^{\frac12}) \|\sqrt{\sigma} (\bu-\bu_h)\|_\Omega \\
& +  C \,\|\bbeta\|_{\infty,\Omega} \left( \sum_{T\in \mct}  \left(\| \bu-\bPi \bu\|_{T}^2+ h_T^2 \| \nabla(\bu-\bPi \bu)\|_{T}^2\right) \right)^{\frac{1}{2}}\,.
\end{alignat*}
\end{theorem}

\begin{proof}
Using the surjectivity of the divergence operator as a mapping from
$\bH^1_0(\Omega)$ to $L^2_0(\Omega)$ there exists $\bv_p
\in \bH^1_0(\Omega)$ such that $\dive \bv_p = P_\ell p - p_h$ and
\begin{equation}\label{aux567}
\|\bv_p\|_{1,\Omega} \leq C \|P_\ell p - p_h\|_\Omega.
\end{equation}
It follows  from \eqref{aux567} and \eqref{commutative} that
\[
\|P_\ell p - p_h\|_\Omega^2 = (P_\ell p - p_h,\dive\bv_p) =
(P_\ell p - p_h,\dive \tilde \bPi \bv_p)  =
(p - p_h,\dive \tilde \bPi \bv_p).
\]
If $\bV_h \equiv \bV_{h,k}^{\text{RT}}$ then choose $\tilde \bPi \bv_p \in
\bV_{h,k}^{\text{RT}}$
and if $\bV_h \equiv \bV_{h,k}^{\text{BDM}}$ choose $\tilde \bPi \bv_p \in
\bV_{h,{k-1}}^{\text{RT}} \subset \bV_{h,k}^{\text{BDM}}$.
Using \eqref{fem} and \eqref{weak} we find that
\begin{equation}
(p - p_h,\dive \tilde \bPi \bv_p) = -(\bu - \bu_h, \bbeta \cdot \nabla
\tilde \bPi \bv_p)_h+ \langle (\bbeta \cdot \bn) (\bu - \widehat{\bu_h} ), \tilde \bPi \bv_p
\rangle_h+(\sigma (\bu - \bu_h), \tilde \bPi \bv_p).
\end{equation}
Applying the Cauchy-Schwarz inequality and the stability of the RT
interpolant and of $\bv_p$ we have
\[
-(\bu - \bu_h, \bbeta \cdot \nabla
\tilde\bPi \bv_p)_h + (\sigma (\bu - \bu_h), \tilde \bPi \bv_p) \leq
(\|\bbeta\|_{\infty,\Omega} \sigma_0^{-\frac12} + \sigma^{\frac12})
\|\sqrt{\sigma} (\bu - \bu_h)\|_\Omega \| \bv_p\|_{1,\Omega}.
\]
For the remaining term observe that, by the definition of  $\langle \cdot, \cdot
\rangle_h$, the fact that $\bbeta \cdot \bn$ changes sign on
neighbouring elements and that $(\bu - \widehat{\bu_h} )$ is single
valued on the faces of the triangulation,
\[
\langle (\bbeta \cdot \bn) (\bu - \widehat{\bu_h} ), \tilde \bPi \bv_p
\rangle_h = \langle (\bbeta \cdot \bn) (\bu - \widehat{\bu_h}
), (\tilde \bPi \bv_p - \bv_p) 
\rangle_h.
\]
The right hand side of this equality is bounded using the Cauchy-Schwarz
inequality, the trace inequality \eqref{local-trace} and the
interpolation properties of the RT-interpolant of Lemma
\ref{errorproj} as follows
\begin{alignat*}{1}
& \langle (\bbeta \cdot \bn) (\bu - \widehat{\bu_h}
), (\tilde \bPi \bv_p - \bv_p) 
\rangle_h \\
& \leq   C \|\bbeta\|_{\infty,\Omega} \sum_{T \in\mathcal{T}_h} (h_T^{-\frac12} \|\bu -
\bu_h\|_{T}+ h_T^{\frac12} \|\nabla (\bu -
\bu_h)\|_{T}) h_T^{\frac12}
\|\bv_p\|_{1,T} \\
& \le  C \|\bbeta\|_{\infty,\Omega} \sum_{T \in\mathcal{T}_h} (\|\bu -
\bu_h\|_{T}+\|\bu -\bPi \bu\|_{T}+ h_T\|\nabla (\bu -
\bPi \bu)\|_{T})
\|\bv_p\|_{1,T},
\end{alignat*}
where in the last step we added and subtracted $\bPi \bu$, used
the triangle inequality and the inverse inequality \eqref{inverse}.
We conclude by using \eqref{aux567} .
\end{proof}

The following result is an immediate consequence of 
Theorem \ref{thm:pres} and Corollary \ref{cor:conv} and the
approximation properties of the $L^2$-projection,
\begin{corollary}\label{cor:pres_order}
Assume that $\bV_h=\bV_{h,k}^{\text{RT}}$ and $M_h=M_{h,k}$. Then, there exists 
$\tilde C_{\bbeta,\sigma}>0$ that depends only on the constants in the
bounds of Theorems \ref{thm:pres} and Corollary \ref{cor:conv} such that
\[
\|p - p_h\|_\Omega
\leq \tilde C_{\bbeta,\sigma} h^{k+\frac12} \|\bu\|_{k+1,\Omega}
+ C h^{k+1}|p|_{k+1,\Omega}\,.
\]
For the case in which  $\bV_h=\bV_{h,k}^{\text{BDM}}$ and 
$M_h=M_{h,k-1}$, the following error estimate holds
\[
\|P_{k-1} p - p_h\|_\Omega 
\leq \hat C_{\bbeta,\sigma} h^{k+\frac12} \|\bu\|_{k+1,\Omega} 
\]
and 
\[
\|p - p_h\|_\Omega \leq \hat C_{\bbeta,\sigma} h^{k+\frac12}
\|\bu\|_{k+1,\Omega} + C h^k |p|_{k,\Omega},
\]
where $\hat C_{\bbeta,\sigma}$ depends on the constants in the
bounds of Theorems \ref{thm:pres} and Corollary \ref{cor:conv}.
\end{corollary}
\section{A numerical example}\label{sec:num}
Here we will show some illustrations of the theory developed above
using the analytical solution of example (2) in section \ref{sec:linear}. For
ample qualitative numerical evidence of the performance of this type of method on
physically relevant problems we refer to the references \cite{SL18,SLLL18}.

We consider the domain $\Omega=(0,1) \times (0,1)$ and the solution
\eqref{eq:vel_vortex}-\eqref{eq:press_vortex} of
example (2). We used the package FreeFEM++ \cite{He12} to implement the formulation \eqref{fem} with either the
BDM(1) element and piecewise constant pressures or the RT(1) element
with piecewise affine, discontinuous, pressures. The linear systems
were solved using UMFPACK and the meshes were of Union Jack type. In Tables
\ref{tab:BDM}-\ref{tab:RT} we report the errors of velocities and
pressures in the (relative) $L^2$-norm. We also report the CPU
time. We see that the velocity approximations have identical errors in
the two cases as predicted by Proposition \ref{eq:RTeqBDM}, whereas as expected the BDM(1) approximation has poorer
convergence of the pressure. The RT(1) computation however is more
costly by almost a factor three.

In Table \ref{tab:n} we report the variation of the error on a fixed
mesh with $h=1/40$ and $\sigma=100$. The variable $n$, controlling the
number of vortices, and hence influencing both
$\|\bbeta\|_{W^{1,\infty}(\Omega)}$ and $\|\bu\|_{H^2(\Omega)}$ is
taken in the set $n \in \{1,2,4,8\}$. We observe (approximately) linear growth in both
velocity and pressures, except for the pressure for the method using the RT
element, where the growth is stronger. For the highest value $n=8$,
all errors are above $15\%$ on this mesh. In Table \ref{tab:sig} we
vary the coefficient $\sigma$ and see that also here the error growth
for decreasing $\sigma$ is by and large linear for the velocities, as
predicted by theory (Corollary \ref{cor:conv}) and
the RT pressure (Corollary \ref{cor:pres_order}). The BDM pressure on the other hand is very robust
with respect to variations in $\sigma$, but much larger than the
RT-pressure. It starts increasing only for
the smallest value of the parameter, when the pressure errors of the
two approximation spaces are comparable. It follows that for small values
of $\sigma$ the pressure approximation is of similar quality for the
BDM and RT methods.
\begin{table}
\begin{tabular}{|c|c|c|c|}
\hline
$h$ & $\|\bu - \bu_h\|_{L^2(\Omega)}$ & $\|p - p_h\|_{L^2(\Omega)}$ & CPU\\
\hline
$1/10$ &  $0.011\; (-)$ & $0.15\; (-)$ & $0.073s$\\
$1/20$ &  $0.0030\; (1.9)$ & $ 0.074\; (1.0)$ &$0.47s$ \\
$1/40$ &  $0.00087\; (1.8)$ & $ 0.037\; (1.0)$ & $4.7s$\\
$1/80$ &  $0.00031\; (1.5)$ & $0.019\; (1.0)$ & $62.9s$\\
\hline
\end{tabular}
\vspace{5mm}
\caption{Errors for the BDM1/P0 element. $\sigma=100$ $n=1$.}
\label{tab:BDM}
\end{table}
\begin{table}
\begin{tabular}{|c|c|c|c|}
\hline
$h$ & $\|\bu - \bu_h\|_{L^2(\Omega)}$ & $\|p - p_h\|_{L^2(\Omega)}$ & CPU \\
\hline
$1/10$ &  $0.011\; (-)$ & $0.026\; (-)$ & $0.17s$\\
$1/20$ &  $0.0030\; (1.9)$ & $ 0.0060\; (2.1)$  & $1.2s$\\
$1/40$ &  $0.00087\; (1.8)$ & $ 0.0018\; (1.7)$ & $12s$ \\
$1/80$ &  $0.00031\; (1.5)$ & $0.00073\; (1.3)$ & $165s$ \\
\hline
\end{tabular}
\vspace{5mm}
\caption{Errors for the RT1/P1dc element. $\sigma=100$ $n=1$.}
\label{tab:RT}
\end{table}
\begin{table}
\begin{tabular}{|c|c|c|c|c|}
\hline
$n$ & BDM $\|\bu - \bu_h\|_{L^2(\Omega)}$ & BDM $\|p -
                                           p_h\|_{L^2(\Omega)}$ & RT
                                                                  $\|\bu
                                                                  -
                                                                  \bu_h\|_{L^2(\Omega)}$
  &  RT $\|p - p_h\|_{L^2(\Omega)}$\\
\hline
$1$ &   $0.00087$ & $ 0.037$&   $0.00087$ & $ 0.0018$ \\
$2$ &  $0.0048$ & $0.074$ &  $0.0048$ & $0.0058$ \\
$4$ &  $0.031$ & $0.14$ &  $0.031$ & $0.026$ \\
$8$ &  $0.21$ & $0.34$ & $0.21$ & $0.18$\\
\hline
\end{tabular}
\vspace{5mm}
\caption{Errors for the BDM1/P0 element (columns 2 and 3) and RT1/P1dc
  element (columns 4 and 5), $h=1/40$, $\sigma=100$,
  varying $n$.}\label{tab:n}
\end{table}
\begin{table}
\begin{tabular}{|c|c|c|c|c|}
\hline
$\sigma$ & BDM $\|\bu - \bu_h\|_{L^2(\Omega)}$ & BDM $\|p -
                                           p_h\|_{L^2(\Omega)}$ & RT
                                                                  $\|\bu
                                                                  -
                                                                  \bu_h\|_{L^2(\Omega)}$
  &  RT $\|p - p_h\|_{L^2(\Omega)}$\\
\hline
$10^6$&   $0.00061$ & $ 0.037$&   $0.00061$ & $ 0.015$ \\
$100$ &   $0.00087$ & $ 0.037$&   $0.00087$ & $ 0.0018$ \\
$50$ &  $0.0012$ & $0.037$ &  $0.0012$ & $0.0019$ \\
$25$ &  $0.0021$ & $0.037$ &  $0.0021$ & $0.0022$ \\
$10$ &  $0.0051$ & $0.037$ & $0.0051$ & $0.0045$\\
$1$ &  $0.048$ & $0.058$ & $0.048$ & $0.045$\\
\hline
\end{tabular}
\vspace{5mm}
\caption{Errors for the BDM1/P0 element (columns 2 and 3) and RT1/P1dc
  element (columns 4 and 5), $h=1/40$, $n=1$,
  varying $\sigma$.}\label{tab:sig}
\end{table}

\section*{Acknowledgments} The work of Gabriel R. Barrenechea has been funded by the
Leverhulme Trust through the Research Fellowship No. RF-2019-510. Erik
Burman was partially supported by the grant: EP/P01576X/1. Johnny Guzman was partially supported by the grant: NSF, DMS \# 1620100.
\bibliographystyle{abbrv}
\bibliography{references}

\end{document}